\documentclass[10pt]{amsart}
\usepackage{amsmath}
\usepackage{amsthm}
\usepackage{amssymb}
\usepackage{mathrsfs}
\usepackage{hyperref}
\usepackage{color}

\usepackage{psfrag,graphicx}

\usepackage{caption} 

\usepackage{epic,eepic}
\usepackage{color}
\usepackage{ebezier}

\newtheorem{theorem}{Theorem} [section]
\newtheorem*{remark}{Remark}

\newtheorem{lemma}[theorem]{Lemma}

\numberwithin{equation}{section}

\newcommand{\diam}{\operatorname{diam}}

\newcommand{\length}{\operatorname{\length}}

\newcommand{\lenghth}{\operatorname{lenghth}}
\newcommand{\vep}{\operatorname{\varepsilon}}
\newcommand{\SL}{\operatorname{SL}}
\newcommand{\dm}{\operatorname{d}}

\def\I{\operatorname{I}}
\def\II{\operatorname{II}}

\begin{document}
	\title{Margulis-Ruelle inequality for general manifolds}
	\author{Gang Liao  and Na Qiu}

\email{lg@suda.edu.cn}

\email{na.qiu@qq.com}



\thanks{School of Mathematical Sciences, Center for Dynamical Systems and Differential Equations, Soochow University,
Suzhou 215006,  China; This research  was partially supported by NSFC (11701402,  11790274),  IEPJ and NSF (BK 20170327) of  Jiangsu Province}
	
\begin{abstract}
	
In this paper we investigate the Margulis-Ruelle inequality for  general Riemannian manifolds (possibly noncompact and with boundary) and  show that it always holds under integrable  condition.
\end{abstract}

\maketitle
\section{Introduction}

In the theroy of  differentiable dynamical systems,  Margulis-Ruelle inequality \cite{Ruelle} is a fundamental  formulation.  This  inequality establishes   the   bridge  between   the chaotic index metric entropy and the stable index Lyapunov exponent,  being employed  especially  frequently   in  the study  of  hyperbolic behaviors  in various settings (uniformly, nonuniformly  or partially hyperbolic systems).

Let $f$ be a continuous map on a  metric  space $M$.    For  $f$-invariant Borel probability measure $\mu$  and    measurable partition $ \xi $  of  $ M$ with partition  entropy 
$
H_{\mu}(\xi)=\int\log\mu^{-1}(\xi(x))d\mu<\infty
$, 
the metric entropy of $f$ with respect to  $ \xi $ is given by
\begin{eqnarray}\label{entropy}
h_{\mu}(f,\xi)=\lim\limits_{n\to +\infty}\frac{1}{n}H_{\mu}(\bigvee_{i=0}^{n-1}f^{-i}(\xi))
\end{eqnarray}
and   the metric entropy  $h_{\mu}(f)$ of $f$  with respect to $\mu$ is defined as the supremum of 
(\ref{entropy}) over all $\xi$ with finite partition entropy.  

For  $f$  differentiable   on a Riemannian manifold $M$,  relative to   a direction $v\in T_xM$ with $x\in M$,  Lyapunov exponent along $v$  is  given  by the limit
\begin{eqnarray}\label{limit}\lim_{n\to +\infty}\frac{1}{n}\log \|D_xf^nv\|,\end{eqnarray}
which exists  (possibly -$\infty$), for almost every point $x$ of every  $f$-invariant Borel  probability measure $\mu$ by Oseledets theorem \cite{Oseledets, FLQ} if $\int \log^+\|D_xf\| d\mu < \infty$, where $\log^+ t=\max\{\log t, 0\}$.   Moreover, the finite dimension of $M$ brings about exactly the same number of  values  of the limit (\ref{limit}), which we denote as $\lambda_i(f,x)$ $(1\le i\le \dim M)$. 

 When the manifold $M$ is compact and without boundary,  the above two mechanisms could be linked together  as the layout of  Margulis-Ruelle inequality \cite{Ruelle}:  the metric entropy  of  $\mu$ is  not beyond the sum of positive Lyapunov exponents:
\begin{eqnarray}\label{MR}h_{\mu}(f)\leq \int \sum_{\lambda_i(f,x)> 0} \lambda_i(f,x) \,d\mu.\end{eqnarray}

However, for general manifolds,  the availability of (\ref{MR})  is  unkown yet. There are many  spaces which are expected to be taken into account, but not in the previous context, such as  $\mathbb{R}^n$, Lie groups  $\SL(n, \mathbb{R})$, billiard tables,  moduli spaces, etc, where both noncompactness and boundary  lead to the actual  obstructions. 
As the case exhibits,  if  $M$ is noncompact, the inequality (\ref{MR}) possibly doesn't  hold  \cite{Riquelme}.   With  boundary   permissible,    (\ref{MR})  was dealt with  in  (\cite{KSLP}),  involving  systems on  manifolds with finite capacity.  Regarding the two aspects,     in this paper,  we   investigate this problem in  a  general situation, i.e.,  for  manifolds which are  possibly both   noncompact 
and with boundary.

Let $M$ be a Riemannian mainfold,    $f$ a $C^{1}$ map from an open subset $U\subset M$ to $M\setminus \partial M$, and  Borel probability measure $\mu$ on $M\setminus \partial M$ is $f$-invariant: $\mu(f^{-1}(B))=\mu(B)$ for any Borel set $B\subset M\setminus \partial M$.  The general setting is as follows:
\smallskip

(A) {\it Distortion on  $f$}:  there exist  $0<\alpha<1$,\,$C, a>1$ such that for any $x\in U,\,\,y\in B(x, \min\{1,\,\,d^a(x, \partial M) \})$,  
   $$\frac{\big{|}\|D_xf\|-\|D_yf\|\big{|}}{d(x,y)^{\alpha}} \le C  d_0(x, \partial M )^{-a}, $$
where $d_0(x, \partial M )=\min\{d(x, \partial M ), d^{-1}(x, x_0)\}$  for any preassigned  $x_0\in M$.

\smallskip

(B) {\it Integrability on $\mu$}:
 $$\int_U \max\{\log^+ \|D_xf\|,\,\,|\log d_0(x,\partial M)|,\,\, |\log \rho_{b}(x)|,\,\,\log N_b(x)\,\}\,d\mu<\infty,$$
where  $\rho_b(x)$ for some $b\ge1$  is the regular radius given by the minimum of $\varrho_b(y)$ with  \,$d_0(y, \partial M)\ge d_0(x,\partial M) $ and  $\varrho_b(y)=\max\{0<r\le 1: \,\|D_{w}\exp_{x_1}\|\le b$,\, $ \|D_{x_2}\exp_{x_1}^{-1}\|\le b,\quad \forall \,\,x_i\in B(y,r),\,\,\,i=1,2,\, \,w=\exp^{-1}_{x_1}(x_2)\}$;   $N_b(x)$ is the regular tankage given by    the minimal cardinality among   covers of $\{y: d_0(y, \partial M)\ge  d_0(x, \partial M)\}$ by balls with  regular radius.

\begin{remark} The properties on  $d_0$ is independent of the choice of $x_0$.  In  fact,  $d^{-1}(x,x_0)$ represents the distance of $x$ relative to the infinity: ``away from $x_0$'' is equavlent to ``close to the infinity".

	\begin{remark}
	
  (A)(B) are naturally satisfied if $M$ is compact and without boundary.
	 In (A), we adopt $\|D_xf\|-\|D_yf\|$ instead of $\|D_xf-D_yf\|$, which avoids the assumptions on the  local trivialization of manifolds as in  \cite{KSLP} since we just compare real numbers $\|D_xf\|$ and $\|D_yf\|$.   In (B), the first one is a natrual property guaranteeing the exisitence of Lyapunov 
	exponents in Oseledets theorem,  the next  three  could be verified in ususl analytic settings with some nondegenerate polynomial estimates,  for example,   for any measure absolute to Lebesgue measure such that 
	\begin{itemize}	
		\item the density  decreases polynomially  of order  bigger than $\dim M$ near  infinity   and  increases polynomially  of order smaller than $\dim M$  near  boundary;\\ 
		\item the  regular radius decreases polynomially and   the regular tankage  increases  polynomially    near   infinity and  boundary.  
	\end{itemize}	
		
	\end{remark}

\end{remark}
\begin{theorem}\label{MR Thm}
Let $M$ be a Riemannian manifold and $f$  a $C^{1}$  map  from an open subset $U\subset M$ to $M\setminus \partial M$  satisfying (A). Then  for any $f$-invariant Borel probability measure satisfying $(B)$,   Margulis-Ruelle inequality  holds true.

\end{theorem}

The difficulties to establish (\ref{MR}) mainly  arise from four  aspects: firstly, the norm of derivarives may grow to the infinity;  secondly, the regular neighborhoods on which $f$ behaves like diffeomorphisms may be destroyed by the boundary; thirdly,   the distortion of manifold has no uniform  bounds;   and  fourthly,  the tankage may increase very largely when approaching the infinity and boundary.   To overcome these difficulties,  we  take advantage   of  partitions  which are adaptive   to  the above  items  and  also we show the finiteness of  such partitions' entropy in Section \ref{partition}    for the feasibility of later dynamical entropy computation.   In Section \ref{ evolutions}, we  carry on analysis of entropy for  evolutions, with respect to the  (singular) part near boundary and infinity, and the (regular)  part away from boundary and infinity.  In the present situation,  we adopt  infinite countable partition instead of finite partition as in \cite{KSLP}, which will make the  integrable condition  of Theorem \ref{MR Thm}  more directly applicable to the proof  of (\ref{MR}).

\section{Construction of  partitions and partition entropy}\label{partition}

For convenience, we consider the minimum between $d_0$ and $1$, which we  denote by $d_*$. Notice that the properties in (A)(B) concerning $d_0$ also hold for $d_*$.

Consider iterations $f^m$ with  $m\in \mathbb{N}$.  Associated to the influences  of noncompactness and boundary,   we define for any $n\in \mathbb{N}$, 
\begin{eqnarray*} A_{n}=\Big{\{}x\in U: &&\Pi_{j=0}^{m-1}\|D_{f^j(x)}f\|^*\le 2^{n}, \,\,\, \Pi_{j=0}^{m-1}d_*(f^j(x),\partial M)\ge 2^{-n},\\[2mm] &&\Pi_{j=0}^{m-1}\rho_b(f^j(x))\ge 2^{-n},\,\,\,\Pi_{j=0}^{m-1}N_b(f^j(x))\le 2^{n}  \Big{\}},\end{eqnarray*}
where $t^*=\max\{t,\,1 \}$ for $t\in \mathbb{R}$. 

It holds that $A_1\subset A_2\subset \cdots$ and $\lim_{n\to +\infty}\mu(A_n)=1$ by (B).  We are going to  construct  partitions with  box-like elements whose  sizes are   given differently   according to    $A_n$. 

Denote $\dm=\dim M$.   Relative to any $x\in M\setminus \partial M$,  under  an orthonormal basis  $\{e_1,\cdots, e_{\dm}\}$ of $T_xM$,  a box   centered at $y$ with sides $2a_i>0$ $(1\le i\le \dm)$  is understood as    $$\exp_x(\Gamma(\exp^{-1}_x(y); a_1,\cdots, a_{\dm})):=\exp_x(\{\exp^{-1}_x(y)+\sum_{1\le i\le \dm} t_ia_ie_i,\quad -1\le t_i\le 1\}).$$
   In particular, for any $a>0$, we denote   $\Gamma_x(a)$ as the cube $\exp_x(\Gamma(0; a,\cdots, a))$ and for any  integer $l\ge 1$,  $\mathcal{Y}(x,a,l)$
 is defined as a partition of  $\Gamma_x(a)$ consisting of subcubes $$\exp_x(\{ \sum_{1\le i\le \dm} \frac{k_ia}{2^l}e_i+\sum_{1\le i\le \dm} t_i\frac{a}{2^l}e_i:\,\, -1\le t_i\le 1\}),\,\,  1-2^l\le k_i\le 2^l-1,\,\,  \frac{k_i+1}{2}\in \mathbb{Z}.$$

For any $\vep>0$ and a subset $M_1\subset M$,  denote  $F(M_1,\vep)$ as a  subset  of  $M_1$ with  maximal cardinality in sense that  any $x,y\in M_1$ satisfy  $d(x,y)> \vep$.  As a matter of fact,   $M_1\subset \cup_{x\in F(M_1,\vep) }\bar{B}(x,\vep)$, where $\bar{B}(x,\vep)$ is  the closure of ball $B(x,\vep)$.   Noting that   $\exp_x^{-1}(B(x,\rho_b(x))= B_{\mathbb{R}^{\dm}}(0, \rho_b(x))$,   by the control in  regular scale  there  exists a constant $C_{0,1}>0$ such that for any  $x\in M\setminus \partial M$,  $0<r\le \rho_b(x)$,  one  has $$\#F(B(x, r),\vep)\le  C_{0,1}( \lceil\frac{r}{\vep}\rceil)^{\dm },\quad \forall\,\vep>0.$$

         Choose  $l_1  \in \mathbb{N}$  such that
$$l_1>a,\quad \frac{1}{2^{n(l_1-2m)}}<\frac{1}{\sqrt{\dm}  2^{n+1}},\quad  C\big{(}\frac{1}{2^{n(l_1-2m)}}\big{)}^{\alpha} 2^{an}<2^{\frac{1}{m}}-1,\quad \forall n\in \mathbb{N},$$  
 and define    \begin{eqnarray*}\vep_{n}&=&\frac{1}{\sqrt{\dm} 2^{nl_1}},\quad \forall\,n\in \mathbb{N}.
\end{eqnarray*}

Fix $n\in \mathbb{N}$.  In order to cover $A_n$, we  take  $F(A_n, \vep_{n})=\{x_1^{(n)},\cdots, x_{k_n}^{(n)}\}$ with $k_{n}=\#F(A_n, \vep_{n})$. It holds that   $\cup_{1\le i\le k_n}\Gamma_{x_i^{(n)}}(\vep_n)\supset A_n$    since  $\Gamma_{x_i^{(n)}}(\vep_n) \supset \bar{B}(x_i^{(n)}, \vep_n )$.  Moreover,    $$\Gamma_{x_i^{(n)}}(\vep_n)\subset \bar{B}(x_i^{(n)}, \,\sqrt{\dm} \vep_n)=\bar{B}(x_i^{(n)},\, \frac{1}{2^{nl_1}}),\quad  \forall\,1\le i\le k_n.$$
We point out that the closure of any ball $B$ centered at $x\in A_n$ with radius $\le \sqrt{\dm}\vep_n$ intersects  at most   $(4\lceil b\dm\rceil )^{\dm}$ such  $\Gamma_{x_i^{(n)}}(\vep_n)$.  Indeed, if not, then the number of   $x_i^{(n)}$  in $\bar{B}(x, 2\sqrt{\dm}\vep_n)\,\subset \Gamma_{x}(2\sqrt{\dm}\vep_n)$
 $(\subset \bar{B}(x,2\dm\vep_n)\subset B(x,\frac{1}{2^n})\subset  B(x,\rho_b(x)) )$  is more than $(4\lceil b\dm\rceil )^{\dm}$,   which leads to  at least two of them staying in the same element of  the equi-partition of  $\Gamma_{x}(2\sqrt{\dm}\vep_n)$ by cubes with sides  $\frac{2\cdot 2\sqrt{\dm}\vep_n}{4\lceil b\dm\rceil}$,   thus  their  distance on $M$ not bigger than  $ \frac{2\cdot 2\sqrt{\dm}\vep_n}{4\lceil b\dm\rceil}\sqrt{\dm}b\le\vep_n$,  that  contradicts the choice of  $x_i^{(n)}$.

To obtain a partition of $A_n$, we define 
\begin{eqnarray*}B_{n;1}&=&A_n\cap \Gamma_{x_1^{(n)}}(\vep_n),\\[2mm]
	B_{n;2}&=& A_n\cap \Gamma_{x_2^{(n)}}(\vep_n)\setminus B_{n;1},\\[2mm]
		B_{n; 3}&=& A_n\cap \Gamma_{x_3^{(n)}}(\vep_n)\setminus (\cup_{1\le i\le 2}B_{n;i}),\\[2mm]
	& \vdots& \\[2mm]
	B_{n; k_{n}}&=& A_n\cap \Gamma_{x_{k_n}^{(n)}}(\vep_n)\setminus (\cup_{1\le i\le k_n-1}B_{n;i}).
\end{eqnarray*} 
 For $s\ge n+1$,  taking  $F(A_s\setminus A_{s-1}, \vep_{s})=\{x_1^{(s)},\cdots, x_{k_s}^{(s)}\}$ with $k_{s}=\#F(A_s\setminus A_{s-1}, \vep_{s})$,  define 
\begin{eqnarray*}B_{s;1}&=&(A_s\setminus A_{s-1})\cap \Gamma_{x_1^{(s)}}(\vep_s),\\[2mm]
	B_{s;2}&=&(A_s\setminus A_{s-1})\cap \Gamma_{x_2^{(s)}}(\vep_s) \setminus B_{s;1},\\[2mm]
	& \vdots& \\[2mm]
	B_{s; k_{s}}&=&  (A_s\setminus A_{s-1})\cap \Gamma_{x_{k_s}^{(s)}}(\vep_s)\setminus (\cup_{1\le i\le k_s-1}B_{s;i}).
\end{eqnarray*}
Then  we may get  two measurable partitions
\begin{eqnarray*} \mathcal{P}&=&\Big{\{}B_{s, i}:\,\,1\le i\le k_s,\,s\ge n\Big{\}}, \\[2mm] \mathcal{\tilde{P}}&=&\Big{\{}  E_s=\cup_{1\le i\le k_s}B_{s, i}:\,s\ge n\Big{\}}.\end{eqnarray*}
In fact, one may see that   $E_n=A_n$ and $E_{s}=A_s\setminus A_{s-1}$ for $s\ge n+1 $.

We continue  by using $l$ to partition each $B_{s;i}$ into smaller sets.   To be precise,  denote  $\mathcal{Y}(x_i^{(s)},\vep_s,l)=\{\Gamma_1^{(s,i,l)}, \Gamma_2^{(s,i,l)},\cdots, \Gamma_{2^{l\dm }}^{(s,i,l)}\}$ and  let 
\begin{eqnarray*}B_{s;i,1}&=&B_{s; i}\cap \Gamma_1^{(s,i,l)},\,\\[2mm]
	B_{s;i,2}&=&B_{s; i }\cap \Gamma_2^{(s,i,l)}\setminus B_{s;i,1},\\[2mm]
	& \vdots& \\[2mm]
	B_{s;i, 2^{l\dm}}&=& B_{s; i}\cap \Gamma_{2^{l\dm}}^{(s,i,l)} \setminus (\cup_{1\le j\le 2^{l\dm }-1}B_{s;i, j}).
\end{eqnarray*}  
Then,  for any $l\ge 1$,   one   can obtain  a partition  as the following
\begin{eqnarray*} \mathcal{P}_n^{(l)}=&&\Big{\{} B_{s; i,j}:\,\,1\le i\le k_s,\,1\le j\le2^{l\dm},\,\,s\ge n\Big{\}}.\end{eqnarray*}
It is obvious that $ \mathcal{\tilde{P}} \prec \mathcal{P}\prec \mathcal{P}_n^{(l)}\,\mod 0\,$ for any $l\ge 1$.   For convenience of statement, we denote $\mathcal{P}_n^{(0)}= \mathcal{P}$.

In the  above construction, $n$ stands for the regularity relative to  the bounday and infinity; $l$ is used to make the size of partitions arbitrarily small. The   choice of $l$ will be confirmed   in the estimate of  entropy with respect to  evolutions.

Note that  $A_s$ is covered by at most $2^s$ balls with regular size (which is $\le 1$) and    each of them can be  further covered by at most $C_{0,1}(\lceil \frac{4}{\vep_s}\rceil)^{\dm} $ balls  with radius $\frac{\vep_s}{4}$ (each such ball is   contained in a ball centered at $A_s$ with radius $\frac{\vep_s}{2}$),   thus  $k_s\le 2^s\cdot C_{0,1}(\lceil \frac{4}{\vep_s}\rceil)^{\dm}  $.   Moreover,  each element of $\mathcal{P}_n^{(0)}$ is partitioned by $2^{l\dm }$ cubes  in the constructoin of $\mathcal{P}_n^{(l)}$.  Thus,  for any $s\ge n$, 
\begin{eqnarray*}&&\#\Big{\{} B_{s; i,j}:\,\,1\le i\le k_s,\,1\le j\le 2^{l\dm}\Big{\}}\le2^s \cdot   C_{0,1}(\lceil \frac{4}{\vep_s}\rceil)^{\dm} \cdot 2^{l\dm }\\[2mm] &&=  C_{0,1} (\lceil 4\sqrt{\dm}\rceil )^{\dm}2^{s(1+l_1\dm)+l\dm}:=C_02^{s(1+l_1\dm)+l\dm}.\end{eqnarray*}

   Before going any  further,   we  first show the finiteness of partition entropy $H_{\mu}(\mathcal{P}_n^{(l)}) $ so that the constructed partitions are feasible in the computation of  dynamical entropy. 
\begin{lemma}\label{finite}
$H_{\mu}(\mathcal{P}_n^{(l)})<\infty$ for any integer $l\ge 0$.
\end{lemma}

\begin{proof}

Denote \begin{eqnarray*}p(x)=\max\Big{\{} &&\Pi_{j=0}^{m-1}\|D_{f^j(x)}f\|^*,\,\, \Pi_{j=0}^{m-1}d_*^{-1}(f^j(x),\partial M),\\[2mm] &&\,\,\Pi_{j=0}^{m-1}\rho^{-1}_b(f^i(x)), \,\,\Pi_{j=0}^{m-1}N_b(f^j(x))\quad \Big{\}}.\end{eqnarray*}  Note that  if $x\in B_{s; i,j }$ with $s\ge n+1$, then $x\notin A_{s-1}$, so   $p(x)> 2^{s-1}$. It follows that
\begin{eqnarray}
		\int \log p(x)d\mu(x)
	&=&\sum_{s=n}^{+\infty}\int_{E_s} \log p(x)d\mu(x)\nonumber \\[2mm]
\label{E_n bound}	&\ge&  \int_{E_n} \log p(x)d\mu(x)	+\sum_{s=n+1}^{+\infty}\int_{E_s} \log 2^{s-1} d\mu(x).
\end{eqnarray}
Together with  (B), we  deduce 
$$\sum_{s=n}^{+\infty}\mu(E_s)s<+\infty.$$
Therefore,   \begin{eqnarray*}H_{\mu}(\mathcal{P}_n^{(l)}) &\le& \sum_{s=n}^{+\infty} \mu(E_s)\log (C_0 2^{s(1 +l_1\dm)+l\dm})<+\infty. \end{eqnarray*}
\end{proof}

 \section{Estimates on  entropy for dynamical evolutions}\label{ evolutions}

This section is devoted to  the estimate of  metric entropy of $g$ by using the constructed partitions $\mathcal{P}_n^{(l)}$.  By Lemma \ref{finite}, the entropy of  $\mathcal{P}_n^{(l)}$ is finite,  thus 
\begin{eqnarray*}
	h_{\mu}(g,\mathcal{P}_n^{(l)})&=&\lim_{t\to+\infty}\frac{1}{t}H_{\mu}(\bigvee_{i=0}^{t-1}g^{-i}(\mathcal{P}_n^{(l)}))\\[2mm]
	&=& \lim_{t\to+\infty}\frac{1}{t}\Big{(}H_{\mu}(\mathcal{P}_n^{(l)})+\sum_{i=1}^{t-1}H_{\mu}(g^{-i}(\mathcal{P}_n^{(l)})\mid \bigvee_{j=0}^{i-1}g^{-j}(\mathcal{P}_n^{(l)}))\Big{)}\\[2mm]
	&=& \lim_{t\to+\infty}\frac{1}{t}\Big{(}\sum_{i=1}^{t-1}H_{\mu}(g^{-i}(\mathcal{P}_n^{(l)})\mid \bigvee_{j=0}^{i-1}g^{-j}(\mathcal{P}_n^{(l)}))\Big{)}\\[2mm]
	&\le& H_{\mu}(g^{-1}(\mathcal{P}_n^{(l)})\mid \mathcal{P}_n^{(l)}).
\end{eqnarray*}
Furthermore,  it holds that 
 \begin{eqnarray*}
 H_{\mu}(g^{-1}(\mathcal{P}_n^{(l)})\mid \mathcal{P}_n^{(l)})&=& H_{\mu}(g^{-1}( \mathcal{P}_n^{(l)})\vee  g^{-1}(\mathcal{\tilde{P}}) \mid \mathcal{P}_n^{(l)}) \\[2mm]
&=& H_{\mu}(g^{-1}(\mathcal{\tilde{P}}) \mid \mathcal{P}_n^{(l)})+ H_{\mu}(g^{-1}(\mathcal{P}_n^{(l)}) \mid \mathcal{P}_n^{(l)}\vee g^{-1}(\mathcal{\tilde{P}}))\\[2mm]
&:=& \I+\II.
\end{eqnarray*}

 {\bf (1) Estimate on $\I$}.
Observe that by Lemma \ref{finite},  \begin{eqnarray}\label{I finite}\I\le H_{\mu}(g^{-1}(\mathcal{\tilde{P}}))=H_{\mu}(\mathcal{\tilde{P}})\le H_{\mu}(\mathcal{P}_n^{(0)})<+\infty.\end{eqnarray}
Moreover,   $\mu( E_n)\to 1$ as $n\to\infty$,  thus,  by (\ref{E_n bound}) and the integrability of $\log p(x)$, 
$$\lim_{n\to \infty}\sum_{s=n+1}^{+\infty}\int_{E_s} \log 2^{s-1} d\mu(x)=0.$$
So,  taking  $n$ large enough,   one can make  
$$\sum_{s=n+1 }^{+\infty}\mu(E_s)\log 2^{sl_1\dm}<1.$$
Although we can obtain $\I$ is finite by (\ref{I finite}),   the partition entropy $H_{\mu}(\mathcal{P}_n^{(0)})$ depending  on $g=f^m$   may be very large as $m$ increases.    For the estimate of dynamical entropy,   by taking $l$ large,  we are going to deduce a controllable linear  bound for $\I$  relative to the increase  of  $m$.   

 Notice that  \begin{eqnarray*}
	\I&=&\sum_{B\in \mathcal{P}_n^{(l)}}\sum_{s=n}^{+\infty} -\mu(g^{-1}(E_s)\cap B)\log \frac{\mu(g^{-1}(E_s)\cap B)}{\mu(B)}\\[2mm]
&=&\int\big[\sum_{s=n}^{+\infty}-\frac{\mu(g^{-1}(E_s)\cap \mathcal{P}_n^{(l)}(x))}{\mu(\mathcal{P}_n^{(l)}(x))}\log\frac{\mu(g^{-1}(E_s)\cap \mathcal{P}_n^{(l)}(x))}{\mu(\mathcal{P}_n^{(l)}(x))}\big]d\mu(x)\\[2mm]
&=&\sum_{s=n}^{+\infty}\int\big[-\frac{\mu(g^{-1}(E_s)\cap \mathcal{P}_n^{(l)}(x))}{\mu(\mathcal{P}_n^{(l)}(x))}\log\frac{\mu(g^{-1}(E_s)\cap \mathcal{P}_n^{(l)}(x))}{\mu(\mathcal{P}_n^{(l)}(x))}\big]d\mu(x).
\end{eqnarray*}
Since $\I\mid_{l=0}\,<+\infty$, there exists $L\ge n$ such that
\begin{eqnarray*}\sum_{s=L+1}^{+\infty}\int\big[-\frac{\mu(g^{-1}(E_s)\cap \mathcal{P}_n^{(0)}(x))}{\mu(\mathcal{P}_n^{(0)}(x))}\log\frac{\mu(g^{-1}(E_s)\cap \mathcal{P}_n^{(0)}(x))}{\mu(\mathcal{P}_n^{(0)}(x))}\big]d\mu(x)<1.
\end{eqnarray*}
Recalling that  $\mathcal{P}_n^{(0)}\prec  \mathcal{P}_n^{(l)}$ for every $l\ge 1$, by the convexity of $-x\log x$, we have
\begin{eqnarray*}
	&&\int\big[\sum_{s=L+1}^{+\infty}-\frac{\mu(g^{-1}(E_s)\cap \mathcal{P}_n^{(l)}(x))}{\mu(\mathcal{P}_n^{( l)}(x))}\log\frac{\mu(g^{-1}(E)\cap \mathcal{P}_n^{(l)}(x))}{\mu(\mathcal{P}_n^{(l)}(x))}\big]d\mu(x)\\[2mm]
	&=&\sum_{s=L+1}^{+\infty}\int\big[-\frac{\mu(g^{-1}(E_s)\cap \mathcal{P}_n^{(l)}(x))}{\mu(\mathcal{P}_n^{(l)}(x))}\log\frac{\mu(g^{-1}(E_s)\cap \mathcal{P}_n^{(l)}(x))}{\mu(\mathcal{P}_n^{(l)}(x))}\big]d\mu(x)\\[2mm]
	&\le & \sum_{s=L+1}^{+\infty}\int\big[-\frac{\mu(g^{-1}(E_s)\cap \mathcal{P}_n^{(0)}(x))}{\mu(\mathcal{P}_n^{(0)}(x))}\log\frac{\mu(g^{-1}(E_s)\cap \mathcal{P}_n^{(0)}(x))}{\mu(\mathcal{P}_n^{(0)}(x))}\big]d\mu(x)\\[2mm]
	&<& 1.
\end{eqnarray*}

For  the expression of $\I$,   the term $ \mathcal{P}_n^{(l)}$ stays below in   the form of  conditional entropy.  In order to make $\I$ small,   we  can  let $\mathcal{P}_n^{(l)}$  much finer, i.e., let $l$  large.  To be precise, we let $l$ satisfy
\begin{eqnarray*}
 \frac{1}{2^{l-2-m(L+1)}}<\frac{\min\{1-2^{-\frac{1}{m}},\,2^{-1-\frac{1}{m}}\}}{2^{L}},\quad\quad
\big{(}\frac{1}{2^{l-2-m(L+1)}}\big{)}^{\alpha}C2^{aL}<2^{\frac{1}{m}}-1.
\end{eqnarray*}

For  $B=B_{k;i,j}$  contained in  some  $E_k$,  denote $s_B=\min\{s: g(B)\cap E_s\neq \emptyset\}$.  We split  the discussion into two cases:

\smallskip

(i) $s_B>L$. Then
\begin{eqnarray*}
\sum_{s=n}^{+\infty} -\mu(g^{-1}(E_s)\cap B)\log \frac{\mu(g^{-1}(E_s)\cap B)}{\mu(B)}=\sum_{s=L+1}^{+\infty} -\mu(g^{-1}(E_s)\cap B)\log \frac{\mu(g^{-1}(E_s)\cap B)}{\mu(B)}.
\end{eqnarray*}
So,
\begin{eqnarray*}
	&&\sum_{s_B>L}\sum_{s=n}^{+\infty} -\mu(g^{-1}(E_s)\cap B)\log \frac{\mu(g^{-1}(E_s)\cap B)}{\mu(B)}\\[2mm] &\le&  \int\big[\sum_{s=L+1}^{+\infty}-\frac{\mu(g^{-1}(E_s)\cap \mathcal{P}_n^{(l)}(x))}{\mu(\mathcal{P}_n^{(l)}(x))}\log\frac{\mu(g^{-1}(E_s)\cap \mathcal{P}_n^{(l)}(x))}{\mu(\mathcal{P}_n^{(l)}(x))}\big]d\mu(x)<1.
\end{eqnarray*}

\smallskip

(ii) $s_B\le L$. 

\begin{lemma} \label{g(B) including}  There exists a constant integer  $C_1>0$ such that  $g(B)\subset A_{s_B+C_1m}$.

\end{lemma}

\begin{proof}  Step 1:  $\diam(g(B))\le   \frac{1}{2^{l-2}}.$

	 By definition,   $B\subset \Gamma_j^{(k,i,l)}\subset \Gamma_{x_i^{(k)}}(\vep_k)\subset \bar{B}(x_i^{(k)}, \frac{1}{2^{kl_1}})$.  For any $y\in \bar{B}(x, \frac{1}{2^{kl_1}})$ with  $x\in A_k$, 
\begin{eqnarray*}\big{|}\|D_yf\|-\|D_xf\|\big{|}&\le&  C d(x,y)^{\alpha}d_*^{-a}(x,\partial M)\le  C(\frac{1}{2^{kl_1}})^{\alpha}2^{ak}<2^{\frac{1}{m}}-1,\\[2mm]
	d(f(x),f(y))&\le& d(x,y)\max_{u\in \bar{B}(x, \frac{1}{2^{kl_1}})}\|D_uf\| \le \frac{1}{2^{kl_1}} (2^k+1) \\[2mm] &\le& \frac{1}{2^{kl_1}} 2^{k+1}\le \frac{1}{2^{k(l_1-2)}}. \end{eqnarray*}
By induction,  suppose for $1\le j\le m-1$,  \begin{eqnarray*}&&\big{|}\|D_{f^{j-1}(y)}f\|-\|D_{f^{j-1}(x)}f\|\big{|}\le   C(\frac{1}{2^{k(l_1-2(j-1))}})^{\alpha}2^{ak}<2^{\frac{1}{m}}-1,\\[2mm]
	&&  d(f^{j}(x),f^{j}(y))\le  \frac{1}{2^{k(l_1-2j)}}. \end{eqnarray*}
Then 
one has 
\begin{eqnarray*}&&\big{|}\|D_{f^j(y)}f\|-\|D_{f^j(x)}f\|\big{|}\\[2mm]
	&&\le Cd(f^j(x),f^j(y))^{\alpha}d_*^{-a}(f^j(x),\partial M)\le  C(\frac{1}{2^{k(l_1-2j)}})^{\alpha}2^{ak}<2^{\frac{1}{m}}-1,\\[2mm]
	&&  d(f^{j+1}(x),f^{j+1}(y))\le \frac{1}{2^{k(l_1-2j)}}(2^k+1)\le \frac{1}{2^{k(l_1-2(j+1))}}. \end{eqnarray*}
Hence, 
$$\|D_yg\|\le \Pi_{j=0}^{m-1}(\|D_{f^j(x)}f\|+2^{\frac{1}{m}}-1)\le   \Pi_{j=0}^{m-1} (2^{\frac{1}{m}}\|D_{f^j(x)}f\|^*) \le 2^{k+1},$$
which gives rise to $$\diam(g(B))\le  \diam(B)\max_{y\in \bar{B}(x_i^{(k)}, \frac{1}{2^{kl_1}})}\|D_yg\|\le \frac{2}{2^{kl_1+l}}2^{k+1}\le \frac{1}{2^{l-2}}.$$

Step 2:  Choose  $y\in g(B)\cap E_{s_B}$,  then   $d(f^j(u),f^j(y))\le  \frac{1}{2^{l-2-j(L+1)}}$   for any  $u\in \bar{B}(y, \frac{1}{2^{l-2}} )$ and $0\le j\le m-1 $.

Note that when $d(u, y)\le \frac{1}{2^{l-2}} $, \begin{eqnarray*}\|D_{u}f\|&\le&  \|D_yf\|+d(u,y)^{\alpha}Cd(y,\partial M)^{-a}\\[2mm]
	&\le&  2^{s_B} +(\frac{1}{2^{l-2}})^{\alpha}C 2^{as_B}\\[2mm]
	&\le&   2^{s_B} +(\frac{1}{2^{l-2}})^{\alpha}C 2^{aL}\\[2mm]
	&\le& 2^{s_B}+1\le  2^{s_B+1}\le 2^{L+1}.
\end{eqnarray*}
By induction, suppose for $1\le j\le m-1$,   $d(f^{j-1}(u), f^{j-1}(y))\le \frac{1}{2^{l-2-(j-1)(L+1)}} $ and $\|D_{u_1}f\|\le 2^{L+1} $ for any $u_1$ with $d(u_1,  f^{j-1}(y))\le \frac{1}{2^{l-2-(j-1)(L+1)}}$, then

$$d(f^j(u),f^j(y))\le d(f^{j-1}(u), f^{j-1}(y)) 2^{L+1} \le  \frac{1}{2^{l-2-j(L+1)}}. $$
Moreover, for any $u_1$ with $d(u_1,  f^{j}(y))\le  \frac{1}{2^{l-2-j(L+1)}}$,
\begin{eqnarray*}\|D_{u_1}f\|&\le&  \|D_{f^j(y)}f\|+d(u_1, f^j(y))^{\alpha}Cd(f^j(y),\partial M)^{-a}\\[2mm]
	&\le&  2^{s_B} +(\frac{1}{2^{l-2-j(L+1)}})^{\alpha}C 2^{as_B}\\[2mm]
	&\le&   2^{s_B} +(\frac{1}{2^{l-2-j(L+1)}})^{\alpha}C 2^{aL}\\[2mm]
	&\le& 2^{s_B}+1\le  2^{L+1}.
\end{eqnarray*}

Step 3:  $u\in A_{s_B+C_1m}$  for some constant  integer $C_1>0$.  

We need check the properties of $d_*,  \|\,\|^*, \rho_b $ and $N_b$ for $u$.

$1^{\circ}$ $\Pi_{j=0}^{m-1}d_*(f^j(u),\partial M)\ge   2^{-s_B-1}$. 

 By the choice of $l$, for $0\le j\le m-1$,  \begin{eqnarray*} && d(f^j(u), \partial M)\ge d(f^j(y), \partial M)- d(f^j(u), f^j(y)) \ge 2^{-\frac{1}{m}}d(f^j(y), \partial M),\\[2mm]
&&d(f^j(u),x_0)\le d(f^j(y),x_0)+d(f^j(u), f^j(y))\le 2^{\frac{1}{m}}\max\{d(f^j(y),x_0), 1\},
\end{eqnarray*}
which implies that \begin{eqnarray*}\label{d} d_*(f^j(u), \partial M) \ge 2^{-\frac{1}{m}}d_*(f^j(y, \partial M).\end{eqnarray*}
So, $$\Pi_{j=0}^{m-1}d_*(f^ju),\partial M)\ge (2^{-\frac{1}{m}})^m\Pi_{i=0}^{m-1}d_*(f^j(z),\partial M) \ge   2^{-s_B-1}.$$

$2^{\circ}$ $\Pi_{j=0}^{m-1}\|D_{f^j(u)}f\|^*\le   2^{s_B+1}$. 

$$\big{|}\|D_{f^j(u)}f-\|D_{f^j(y)}f\|\big{|}\le  d(f^j(u), f^j(y))^{\alpha}Cd(f^j(y),\partial M)^{-a}\le  2^{\frac{1}{m}}-1,$$
which gives
\begin{eqnarray*}\label{Df} \|D_{f^j(u)}f\|^*\le 2^{\frac{1}{m}}\|D_{f^j(y)}f\|^*.\end{eqnarray*} 
So,  $$\Pi_{j=0}^{m-1}\|D_{f^i(u)}f\|^*\le   (2^{\frac{1}{m}})^m\Pi_{j=0}^{m-1}\|D_{f^j(z)}f\|^* \le 2^{s_B+1}.$$

$3^{\circ}$ $\Pi_{j=0}^{m-1}\rho_b(f^j(u))\ge   2^{-s_B-1}$.  

For any $z$ satisfying $d_0(z, \partial M)\ge  d_0(f^j(u), \partial M)$,  there exists $z'\in \bar{B}(z, \frac{1}{2^{l-2-j(L+1)}})$  satisfying  $d_0(z', \partial M)\ge d_0(f^j(y), \partial M)$.  Since   \begin{eqnarray*}
	&&d(z,z')\le \frac{1}{2^{l-2-j(L+1)}}\le (1-2^{-\frac{1}{m}})\rho_b(f^j(y)),
\end{eqnarray*}
so,    $$\varrho_b(z)\ge 2^{-\frac{1}{m}}\rho_b(f^j(y)).$$
By the arbitrariness of $z$, we have  \begin{eqnarray*}\label{rho}\rho_b(f^j(u))\ge 2^{-\frac{1}{m}}\rho_b(f^j(y)).\end{eqnarray*}
So, $$\Pi_{j=0}^{m-1}\rho_b(f^j(u))\ge (2^{-\frac{1}{m}})^m\Pi_{i=0}^{m-1}\rho_b(f^j(z))\ge   2^{-s_B-1}.$$

$4^{\circ}$ $\Pi_{j=0}^{m-1}N_b(f^j(u))\le  C_{1,0}^m2^{s_B}$ for some  constant  integer $C_{1,0}>0$.  

 If $d_0(f^j(y), \partial M)\le d_0(f^j(u), \partial M)$, it is obvious that  $N_b(f^j(u))\le N_b(f^j(y))$.  Assume $d_0(f^j(y), \partial M)>d_0(f^j(u), \partial M)$.   Then $$\rho_b(f^j(u))\le \rho_b(f^j(y))\le 2^{\frac{1}{m}}\rho_b(f^j(u)).$$
    So,  there exists a constant integer $C_{1,0} >1 $,  such that   any ball  $B(v, \rho_b(f^j(y)))$ with  $d_0(v,\partial M)\ge d_0(f^j(y), \partial M)$ can be covered by  $C_{1,0}$ balls  centered at $B(v, \rho_b(f^j(y)))$ with  radius $\frac{1}{2}\rho_b(f^j(u))\ge 2^{-\frac{1}{m}-1} \rho_b(f^j(y))$, 
which we denote as 
$$B(y_{v,1},  \frac{1}{2}\rho_b(f^j(u)) ), \cdots,  B(y_{v,C_1},  \frac{1}{2}\rho_b(f^j(u)) ).$$ 
  Moreover,  for any $z$ satisfying $d_0(z, \partial M)\ge  d_0(f^j(u), \partial M)$,  there exists $z'$ $ \in \bar{B}(z, \frac{1}{2^{l-2-j(L+1)}})$  $\subset$   $ B(z, \frac{1}{2}\rho_b(f^j(u)))$  satisfying  $d_0(z', \partial M)\ge d_0(f^j(y), \partial M)$.   Thus,  if $\{v: d_0(v,\partial M)\ge d_0(f^j(y), \partial M)\}$ is covered  by   $B(v_1, \rho_b(f^j(y))), \cdots, B(v_{N_b(f^j(y))}, \rho_b(f^j(y)))$,  then $B(y_{v_i,j},  \rho_b(f^j(u)) )$, $1\le i\le N_b(f^j(y))$, $1\le j\le C_{1,0}$, constitute a cover of  $\{z: d_0(z,\partial M)\ge d_0(f^j(u), \partial M)\}$, which gives 
\begin{eqnarray*}\label{N}N_b(f^j(u))\le C_{1,0}N_b(f^j(y)).\end{eqnarray*}  
So,  $$ \Pi_{j=0}^{m-1}N_b(f^j(u))\le  C_{1,0}^m  \Pi_{j=0}^{m-1}N_b(f^j(y)) \le C_{1,0}^m 2^{s_B}.$$

By $1^{\circ}$, $2^{\circ}$, $3^{\circ}$, $4^{\circ}$, 
   denoting $C_1=\lceil \frac{\log C_{1,0}}{\log 2}\rceil$,  we have $u\in A_{s_B+C_1m}$.  Together with  the arbitrariness of $u$, we  deduce that 
 $g(B)\subset A_{s_B+C_1m}$.

 \end{proof}

 Notice that $g(B)\cap E_j=\emptyset$ for any $ j< s_B$,  so 
\begin{eqnarray*}
	&&\sum_{s_B\le L}\sum_{s=n}^{+\infty} -\mu(g^{-1}(E_s)\cap B)\log \frac{\mu(g^{-1}(E_s)\cap B)}{\mu(B)}\\[2mm]
	&=& \sum_{s_B\le L}\,\,\sum_{s=s_B}^{s_B+C_1m} -\mu(g^{-1}(E_s)\cap B)\log \frac{\mu(g^{-1}(E_s)\cap B)}{\mu(B)}\\[2mm]
	&\le&  \sum_{s_B\le L}\,\,(C_1m+1)e^{-1}\mu(B)\le (C_1m+1)e^{-1}.
\end{eqnarray*}
Therefore, $$\I=\Big{(}\sum_{s_B> L}+\sum_{s_B\le L}\Big{)}\sum_{s=n}^{+\infty} -\mu(g^{-1}(E_s)\cap B)\log \frac{\mu(g^{-1}(E_s)\cap B)}{\mu(B)}\le 1+(C_1m+1)e^{-1}.$$

 {\bf (2) Estimate on $\II$}.
We have
\begin{eqnarray*}
&&\II=\sum_{s=n }^{+\infty}\sum_{F\in\mathcal{P}_n^{(l)} }\sum_{B\in\mathcal{P}_n^{(l)} }-\mu(F\cap g^{-1}(E_s)\cap g^{-1}(B))\log\frac{\mu(F\cap g^{-1}(E_s)\cap g^{-1}(B))}{\mu( F\cap g^{-1}(E_s))}\\[2mm]
&=&\sum_{F\in\mathcal{P}_n^{(l)}}\sum_{B\in\mathcal{P}_n^{(l)}, B\subset  E_n }-\mu(F\cap g^{-1}(B))\log\frac{\mu(F\cap g^{-1}(B))}{\mu( F\cap g^{-1}(E_n))}\\[2mm]
&&+\sum_{s=n+1 }^{+\infty}\sum_{F\in\mathcal{P}_n^{(l)}}\sum_{B\in\mathcal{P}_n^{(l)}, B\subset  E_s }-\mu(F\cap g^{-1}(B))\log\frac{\mu(F\cap g^{-1}(B))}{\mu( F\cap g^{-1}(E_s))}\\[2mm]
&:=& \II_1+\II_2.
\end{eqnarray*}

On the one hand,  for any $F\subset \mathcal{P}_n^{(l)}$, if $F$ is in some $E_k$, then $\diam(g(F))\le \frac{2}{2^{kl_1+l}}2^{k+1}\le \frac{1}{2^{l-2}}$, which implies   $\exp_{x_i^{(s)}}^{-1}(g(F)\cap \Gamma_{x_i^{(s)}}(\vep_s))$ is contained in a  cube with sides $\frac{2b}{2^{l-2}}$ whenever $g(F)\cap \Gamma_{x_i^{(s)}}(\vep_s) \not= \emptyset$ . 
Note that    there exist  at most  $k_s$ 
 such $ \Gamma_{x_i^{(s)}}(\vep_s)$,  and each of them is partitioned by cubes  with sides  $\frac{2\vep_s}{2^l}=\frac{2}{2^{sl_1+l}\sqrt{\dm}}$ in the constructoin of $\mathcal{P}_n^{(l)}$.   So, for any $s\ge n+1$, 
\begin{eqnarray*}\#\{B\in\mathcal{P}_n^{(l)}: B\cap g(F)\neq \emptyset, \,\,B\subset  E_s\} \le k_s\Big{(}\lceil (\frac{2b}{2^{l-2}})/(\frac{2}{2^{sl_1+l}\sqrt{\dm}})\rceil+2\Big{)}^{\dm}\le C_2 2^{4sl_1\dm},\end{eqnarray*}
where $C_2=C_{0}(\lceil 4b\sqrt{\dm}\rceil +2)^{\dm}$.   Hence,
\begin{eqnarray*}
\II_2&\le& \sum_{s=n+1}^{+\infty}\sum_{F\in\mathcal{P}_n^{(l)}}\mu(F\cap g^{-1}(E_s))\log (C_22^{4sl_1\dm})\\[2mm]
&=& \sum_{s=n+1 }^{+\infty}\mu(g^{-1}(E_s))\log(C_2 2^{4sl_1\dm})=\sum_{s=n+1 }^{+\infty}\mu(E_s)\log(C_2 2^{4sl_1\dm})<\log C_2+4.
\end{eqnarray*}

On the other hand,
 \begin{eqnarray*}
	\II_1&=& \,\,\sum_{F\in\mathcal{P}_n^{(l)}, F\subset E_n}\,\,\sum_{B\in\mathcal{P}_n^{(l)}, B\subset  E_n}-\mu(F\cap  g^{-1}(B))\log\frac{\mu(F\cap g^{-1}(B))}{\mu( F\cap g^{-1}(E_n))}\\[2mm]
	&&+ \sum_{F\in\mathcal{P}_n^{(l)}, F\subset  \cup_{ s\ge n+1}E_s}\,\,\sum_{B\in\mathcal{P}_n^{(l)}, B\subset  E_n }-\mu(F\cap g^{-1}(B))\log\frac{\mu(F\cap g^{-1}(B))}{\mu( F\cap g^{-1}(E_n))}\\[2mm]
	&:=& \II_{1,1}+\II_{1,2}.
\end{eqnarray*}
For any $F\subset E_s$, $s\ge n+1$,  it holds that
\begin{eqnarray*}\#\{B\in\mathcal{P}_n^{(l)}: B\cap g(F)\neq \emptyset, \,\,B\subset  E_n\} \le k_n\Big{(}\lceil (\frac{2b}{2^{l-2}})/(\frac{2}{2^{nl_1+l}\sqrt{\dm}})\rceil+2\Big{)}^{\dm}\le C_2 2^{4nl_1\dm},\end{eqnarray*}
which implies
\begin{eqnarray*}
	\II_{1,2}&\le&\sum_{F\in\mathcal{P}_n^{(l)},\,F\subset \cup_{s\ge n+1}E_s}\mu(F\cap  g^{-1}(E_n))\log (C_2 2^{4nl_1\dm})\\[2mm]
	&\le &  \mu(\cup_{s\ge n+1}E_s)\log  (C_2 2^{4nl_1\dm})\le \log C_2+4. 
\end{eqnarray*}
Now, it is  left to estimate $\II_{1,1}$, which is exactly the term contributing  Lyapunov exponent. Observe that
\begin{eqnarray*}
	\II_{1,1}
&\le& \sum_{F\in\mathcal{P}_n^{(l)},\,F\subset E_n}\,\log\#\{B\in\mathcal{P}_n^{(l)}, B\subset  E_n : g(F)\cap B\neq \emptyset \}\cdot\mu(F\cap g^{-1}(E_n))\\[2mm]
	&\le&\sum_{F\in\mathcal{P}_n^{(l)},\,F\subset E_n}\,\log\#\{B\in\mathcal{P}_n^{(l)}, B\subset  E_n : g(F)\cap B\neq \emptyset \}\cdot\mu(F).
\end{eqnarray*}

To estimate the last term in the above inequality, we consider a standard argument in advance.   For any   $\beta>0$,   denote 
by $\xi_{\beta}$  a partition of $\mathbb{R}^{\dm}$ into boxes as follows: 
$$ \Big{\{}[q_1\beta, (q_1+1)\beta]\times \cdots \times[q_{\dm}\beta, (q_{\dm}+1)\beta]: q_i\in \mathbb{Z},\,\,1\le i\le \dm\Big{\}}. $$

For any box $\Gamma(x; v_1,\cdots, v_{\dm}; a_1,\cdots, a_{\dm})=\{x+\sum_{1\le i\le \dm} t_ia_iv_i,\quad 0\le t_i\le 1\}$  with unit vectors $v_i$ and sides $a_i>0$   $(1\le i\le \dm)$, denote by $\phi(\Gamma(x; v_1,\cdots, v_{\dm}; a_1,\cdots, a_{\dm}))$ the minimum number of elements in $\xi_1$ whose union covers $\Gamma(x; v_1,\cdots, v_{\dm}; a_1,\cdots, a_{\dm})$.  

\begin{lemma}[Lemma 12.5 of \cite{Mane}]\label{AB}
	There exists a constant $c>0$ such that for any  box $\Gamma(x; v_1,\cdots, v_{\dm}; a_1,\cdots, a_{\dm})$,  $$\phi(\Gamma(x; v_1,\cdots, v_{\dm}; a_1,\cdots, a_{\dm}))\le c\Pi_{i=1}^{\dm} \max\{a_i,1\}.$$ 
\end{lemma}

Denote by  $(D_xg)^{\wedge \kappa}$  the linear map on the $\kappa$-th exterior algebra
of the tangent space $T_xM$ induced by $D_xg$ and let  $\|(D_xg)^{\wedge}\| = \max_{1\le \kappa \le \dm}  \|(D_xg)^{\wedge \kappa}\|  $. 
\begin{lemma} \label{bound of number}There exists $C_3>0$ such that for any  $F\in\mathcal{P}_n^{(l)}$ and $F\subset E_n$, we have for any $x\in F$,
$$\#\{B\in\mathcal{P}_n^{(l)}, B\subset E_n : g(F)\cap B\neq \emptyset \} \le C_3\|(D_xg)^{\wedge}\|.$$
\end{lemma}
\begin{proof}
Note that $F$ is a subset of some  $\Gamma_j^{n,i,l}$.   The image of  $\Gamma_j^{n,i,l}$ by $g$  intersects at most $(4\lceil b\dm\rceil)^{\dm}$  elements $\Gamma_{x_v^{(n)}}(\vep_n)$.       By a scalling of the  size, applying Lemma \ref{AB}, it holds that for  $\Gamma_j^{n,i,l}:=\exp_{x_{u}^{(n)}}(\Gamma(\sum_{1\le i\le \dm}k_{i}\vep_n2^{- l}e_i;  e_1,\cdots, e_{\dm}; $ $ 2\vep_n2^{- l},\cdots,  2\vep_n 2^{- l}):=\exp_{x_{u}^{(n)}}(\Gamma)$, 
\begin{eqnarray*}&&\#\{G\in\xi_{2\vep_n2^{-l}}: (D_{\exp^{-1}_{x_{u}^{(n)}}(x)}(\exp_{x_{v}^{(n)}}^{-1}g\exp_{x_u^{(n)}}))(\Gamma)\cap G\neq \emptyset \}\\[2mm]
	& \le& c\|D_{\exp^{-1}_{x_{u}^{(n)}}(x)}(\exp_{x_{v}^{(n)}}^{-1}g\exp_{x_u^{(n)}})\| \\[2mm]
	&\le& cb^2\|(D_xg)^{\wedge}\|.\end{eqnarray*}
\\
\begin{figure}[h]
	\centering
	\includegraphics[width=7cm]{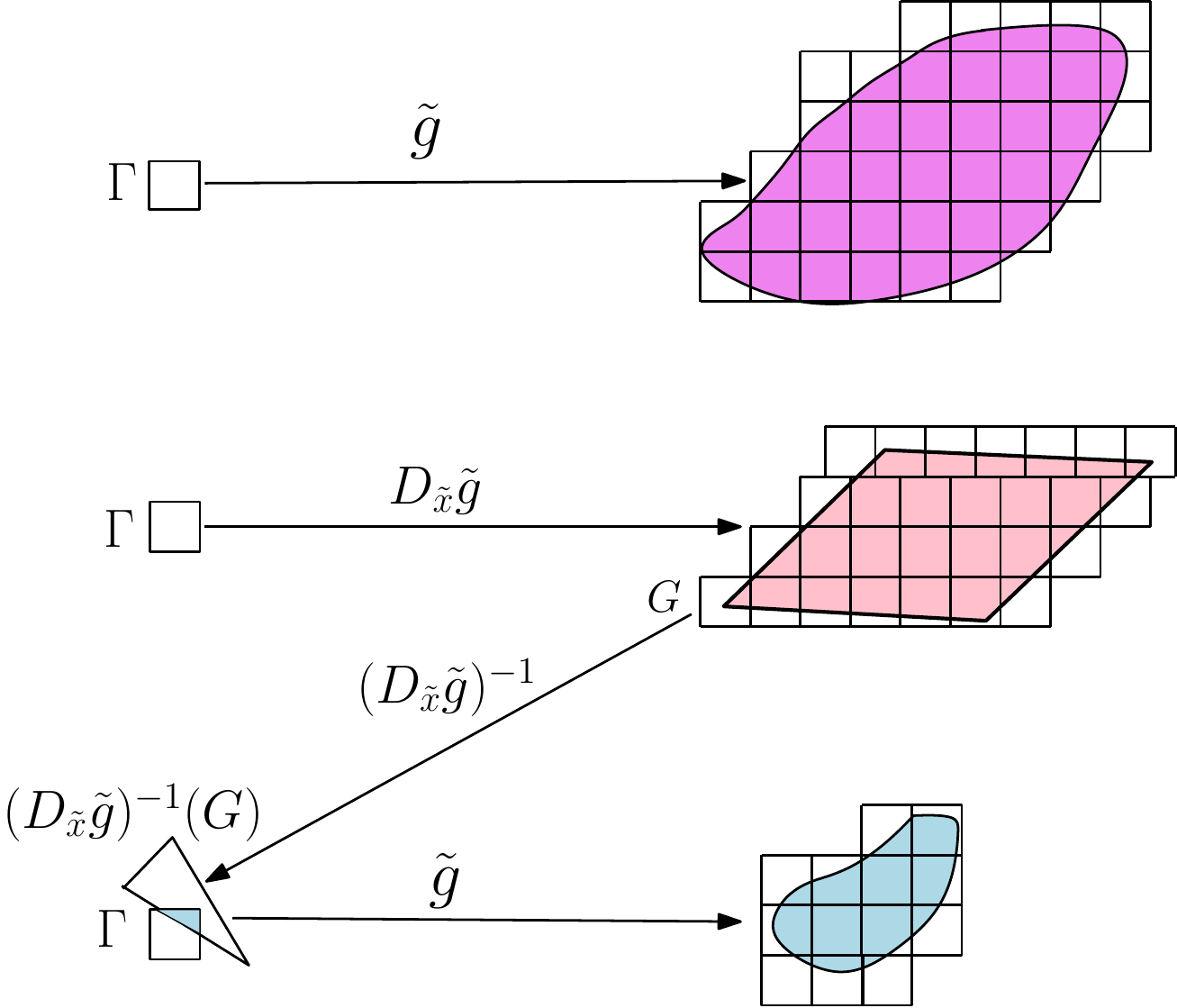}\\
	\caption*{Distortion of boxes}	
\end{figure}

$\,$\\

\noindent{}Denote  $\tilde{x}=\exp^{-1}_{x_{u}^{(n)}}(x)$, $\tilde{g}=\exp_{x_{v}^{(n)}}^{-1}g\exp_{x_u^{(n)}}$.   Take arbitrary  $G\in\xi_{2\vep_n2^{-l}}$ such that $(D_{\tilde{x}} \tilde{g})^{-1}(G)\cap \Gamma\neq \emptyset$. Then for any $y,z\in (D_{\tilde{x}} \tilde{g})^{-1}(G)\cap \Gamma$, we denote $\tau=\frac{y-z}{\|y-z\|}$ and $\sigma$ to be the line segment linking $y$ and $z$. Then,  for the chosen  $n$,  taking $l$ sufficently large,  we have 
\begin{eqnarray*}
d(\tilde{g}(y), \tilde{g}(z))
&=&\int_{\sigma} \|(D_{z}\tilde{g})\tau\|dz \\[2mm] &\le& \int_{\sigma} (\|(D_{\tilde{x}}\tilde{g})\tau\|+1)dz\\[2mm] &\le& \diam(G)+ \lenghth(\sigma)\\[2mm]  &\le& \diam(G)+ \diam(\Gamma)\\[2mm]  &\le& \frac{2}{2^{nl_1+l}}+ \frac{2}{2^{nl_1+l}}=\frac{4}{2^{nl_1+l}}.
\end{eqnarray*}
Hence, 
\begin{eqnarray*}&&\#\{G\in \xi_{2\vep_n2^{-l}}: G\cap \tilde{g}(\Gamma) \neq \emptyset\}\\[2mm]
	&\le&  cb^2\|(D_xg)^{\wedge}\|\cdot (\lceil \frac{\frac{8}{2^{nl_1+l}}}{2\vep_n2^{-l}}\rceil+2)^{\dm}\\[2mm]
	&=& cb^2\|(D_xg)\|^{\wedge}(\lceil 4\sqrt{\dm} \rceil+2)^{\dm}. \end{eqnarray*}
Considering $B$ as the part of the image of $G$ by $\exp_{x_{v}^{(n)}}$  in the construction of $\mathcal{P}_n^{(l)} $ and letting $C_3=cb^2(\lceil4\sqrt{\dm} \rceil+2)^{\dm}(4\lceil b\dm\rceil)^{\dm}$, one gets the estimate of the lemma.
\end{proof}
By Lemma \ref{bound of number}, we have

\begin{eqnarray*}
	\II_{1,1}&\le & \sum_{F\in\mathcal{P}_n^{(l)},\,F\subset E_n}\,\,\int_F \log (C_3\|(D_xg)^{\wedge}\|)d\mu\\[2mm]
	&\le&\log C_3+ \int_{M}  \log \|(D_xg)^{\wedge}\|d\mu\\[2mm]
	&=& \log C_3+ \int_{M}  \log \|(D_xf^m)^{\wedge}\|d\mu.
\end{eqnarray*}

\noindent{\it Proof of Theorem \ref{MR Thm}}. For any $m\in \mathbb{N}$,    sufficently large $n$ and $l$, we have obtained
\begin{eqnarray*}
&&h_{\mu}(f^m,\mathcal{P}_n^{(l)})\le \I+\II_2+\II_{1,2}+\II_{1,1}\\[2mm]
&\le& 1+(C_1m+1)e^{-1}+(\log  C_2 +4)+(\log C_2+4)+\log  C_3+ \int_{M}  \log \|(D_xf^m)^{\wedge}\|d\mu.
\end{eqnarray*}
Letting $l\to +\infty$, we get
\begin{eqnarray*}
&&h_{\mu}(f^m)\\[2mm]
&\le& 1+(C_1m+1)e^{-1}+2(\log  C_2 +4)+ \log C_3+ \int_M  \log \|(D_xf^m)^{\wedge}\|d\mu.
\end{eqnarray*}
Therefore,  
\begin{eqnarray*}
	h_{\mu}(f)=\lim_{m\to +\infty}\frac{1}{m}h_{\mu}(f^m)
	\le C_1e^{-1}+ \int_{M} \sum_{\lambda_i(f, x)> 0} \lambda_i(f, x) \,d\mu.
\end{eqnarray*}
Furthermore,  \begin{eqnarray*}
	h_{\mu}(f)=\lim_{m\to +\infty}\frac{1}{m}h_{\mu}(f^m)
	&\le&  \lim_{m\to +\infty}\frac{1}{m} \int_{M} \sum_{\lambda_i(f^m, x)> 0} \lambda_i(f^m, x) \,d\mu\\[2mm]&=&\int_{M} \sum_{\lambda_i(f, x)> 0} \lambda_i(f, x) \,d\mu.
\end{eqnarray*}
This completes the proof of Theorem \ref{MR Thm}. 
\hfill $\Box$

\end{document}